\documentclass[10pt]{amsart}

\usepackage{geometry,graphicx,amssymb,amsmath,amsbsy,eucal,amsfonts,mathrsfs,amscd,bm}
\usepackage[all]{xy}
\usepackage{tikz}
\usetikzlibrary{arrows,shapes}
\usetikzlibrary{arrows, decorations.markings,fit}
\usetikzlibrary{calc,3d}
\usepackage{tikz-3dplot}
\usepackage{algorithm2e}
\usepackage{listings}
\usepackage[fleqn]{mathtools}
\usepackage{framed}
\newenvironment{mathframed}{\framed%
\allowdisplaybreaks
\vspace*{-\abovedisplayskip}\noindent}{%
\vspace*{-\dimexpr\baselineskip+\topsep}\endframed}

\numberwithin{equation}{section}

\allowdisplaybreaks[3]

\newtheorem{theorem}{Theorem}[section]
\newtheorem{lemma}[theorem]{Lemma}

\newtheorem{proposition}[theorem]{Proposition}
\theoremstyle{definition}
\newtheorem{definition}[theorem]{Definition}
\theoremstyle{remark}

\newcommand{\p}{{\partial}}

\newcommand{\nab}{\nabla}

\newcommand{\mct}{\mathcal{T}_h}

\newcommand{\jump}[1]{\left[\hspace{-0.025in}\left[#1\right]\hspace{-0.025in}\right]}

\newcommand{\Div}{{\rm div}\,}

\newcommand{\pol}{\EuScript{P}}
\newcommand{\bpol}{\boldsymbol{\pol}}
\newcommand{\bld}[1]{\boldsymbol{#1}}
\newcommand{\bt}{\bld{t}}

\newcommand{\bv}{\bld{v}}
\newcommand{\bw}{\bld{w}}
\newcommand{\bs}{\bld{s}}

\newcommand{\bn}{\bld{n}}
\newcommand{\bu}{\bld{u}}

\newcommand{\bV}{\bld{V}}
\newcommand{\mbV}{\mathring{\bV}}
\newcommand{\mY}{\mathring{Y}}
\newcommand{\hY}{\hat{Y}}

\newcommand{\bPi}{\bld{\Pi}}

\newcommand{\bH}{\bld{H}}

\newcommand{\bphi}{\bm \phi}

\newcommand{\calE}{\mathcal{E}}

\newcommand{\bbR}{\mathbb{R}}
\newcommand{\PST}{\mct^{\rm PS}}
\newcommand{\WFT}{\mct^{\rm WF}}
\newcommand{\calS}{\mathcal{S}}
\newcommand{\calT}{\mathcal{T}}
\newcommand{\calF}{\mathcal{F}}
\newcommand{\calQ}{\mathcal{Q}}

\begin{document}

\title[Powell-Sabin and Worsey Farin]{Low-order divergence-free approximations for the Stokes problem on Worsey-Farin and Powell-Sabin splits}


\author[M. Fabien]{Maurice Fabien \textsuperscript{\textdagger}}
\address{\textsuperscript{\textdagger} Division of Applied Mathematics, Brown University, Providence, RI 02912, USA}
\email{fabien@brown.edu}

\author[J. Guzm\'an]{Johnny Guzm\'an\textsuperscript{\textdagger}}
\address{\textsuperscript{\textdagger} Division of Applied Mathematics, Brown University, Providence, RI 02912, USA}
\email{johnny\_guzman@brown.edu}
\thanks{J.~Guzm\'an was partially supported by NSF grant DMS-1913083,}

\author[M. Neilan]{Michael Neilan \textsuperscript{\textdaggerdbl}}
\address{\textsuperscript{\textdaggerdbl}Department of Mathematics,
University of Pittsburgh}
\email{neilan@pitt.edu}
\thanks{M.~Neilan and A.~Zytoon was partially supported by NSF grant DMS-2011733,}

\author[A. Zytoon]{Ahmed Zytoon \textsuperscript{\textdaggerdbl}}
\address{\textsuperscript{\textdaggerdbl}Department of Mathematics,University of Pittsburgh}
\email{AMZ56@pitt.edu}

\maketitle

\begin{abstract}  
We derive low-order, inf-sup stable and divergence-free finite element approximations for the Stokes problem
using Worsey-Farin splits in three dimensions and Powell-Sabin splits in two dimensions. 
The velocity space simply consists of continuous, piecewise linear polynomials, where as
the pressure space is a subspace of piecewise constants with weak continuity properties
at singular edges (3D) and singular vertices (2D).
We discuss implementation aspects that arise when coding the pressure space, and in particular, show that the pressure constraints can be enforced at an algebraic level.
\end{abstract}
\medskip

\smallskip

\date{}

\section{Introduction}
The first inf-sup stable finite element spaces that yield divergence-free approximations for incompressible fluid models on simplicial triangulations were given in the classical paper by Scott and Vogelius  \cite{ScottVogelius85}. Their lowest order pair, defined on a two-dimensional domain, used quartic Lagrage finite elements and piecewise cubic polynomials for the discrete velocity and pressure spaces, respectively; since then, there has been several efforts to find lower order inf-sup stable and conforming finite element spaces that produce divergence-free approximations \cite{qin1994convergence, qin2007stability, arnold1992quadratic, zhang2005,zhang2008p1,Zhang09A,guzman2014conforming, guzman2018inf, ZhangS2011}.

If we confine ourselves to Lagrange elements for the velocity space, then a reduction of polynomial
degree, while preserving stability, can be done on certain splits (i.e., refinements) of a simplicial mesh.
For example, Zhang  proved
stability of the three-dimensional cubic-quadratic pair on Alfeld splits \cite{zhang2005}, and stability
of the three-dimensional quadratic-linear pair on Worsey-Farin splits \cite{ZhangS2011}.
On the other hand, Zhang suggests that piecewise linear Lagrange elements for the velocity are unstable on Worsey-Farin splits \cite[p.244]{ZhangS2011}. One of the main objectives of this paper is to show that indeed the linear case is stable if coupled with the correct pressure space. 

The main goal of this paper is to construct inf-sup stable and divergence-free pairs
on simplicial (split) meshes using a linear-constant velocity-pressure pair on Worsey-Farin splits (3D)
and Powell-Sabin splits (2D).
As far as we are aware, this is the first three-dimensional pair with these
properties using linear Lagrange elements for the discrete velocity space.
The main driver of our construction is an explicit characterization 
of the divergence operator acting on the Lagrange finite element space
developed in  \cite{GuzmanLischkeNeilan, GLN20}. 
Such characterizations naturally lead to the definitions of the discrete pressure
space.  In particular, the pressure space is a subspace of piecewise constants
with weak continuity constraints on singular edges (Worsey-Farin) and singular vertices (Powell-Sabin).

Our results in two dimensions are similar to
those in \cite{zhang2008p1}, 
where Zhang showed stability of a Stokes pair with linear Lagrange elements for the velocity space
on Powell--Sabin splits.  However, the pressure space was not characterized in \cite{zhang2008p1} and as a result, a pressure basis is not explicitly constructed.  Due to the implicit definition of the pressure space, one is forced to solve the resulting algebraic system by the iterated penalty method (IPM).  In contrast, the explicit characterization of the pressure space we provide opens the door to a library of solvers based on the standard mixed formulation.


An advantage of the proposed schemes is their low computational cost and simple velocity spaces, which
are supported in finite element software libraries.  On the other hand, the discrete pressure spaces
are subspaces of piecewise constants with weak continuity properties at singular
edges and singular vertices; these spaces are nonstandard, and in particular are not readily 
available on computational software packages.  Nonetheless, we show that these weak continuity properties
can be enforced entirely at the algebraic level.  One can simply form the algebraic saddle point problem
using the full (unstable) linear-constant pair, and then perform elementary row and column operations
to this system to enforce the weak continuity constraint.  As a result, the proposed discretizations
can be implemented on standard finite element software packages (e.g., FEniCS).

We note that degrees of freedom and commuting projections were given in \cite{GuzmanLischkeNeilan, GLN20}.  In fact, an entire exact sequence of spaces were presented. However, stability (e.g., inf-sup stability) was not shown, which is something we carry out here. 

Moreover, we provide numerical experiments to validate our theoretical results. In particular, we show that the computed solution is satisfying the standard error estimates in mixed formulation. Also, we provide a time comparison between the stiffness matrix modification method described above and the iterated penalty method.

The paper is organized as follows. In the next section we give notation that will be used in the rest of the paper.  In Section \ref{notation} we provide the notation used
throughout the paper and introduce the Stokes problem.
In Section \ref{PS} prove the inf-sup stability of a low-order finite element pair on Powell-Sabin splits. In Section \ref{WF} we prove stability of the analogous three-dimensional pair on Worsey-Farin splits. In Section \ref{implPS} we discuss implementation aspects on Powell-Sabin splits and in Section \ref{implWF} we do the same for Worsey-Farin splits. Finally, in Section \ref{numerical} we provide numerical experiments.

\section{Preliminaries}\label{notation}
In this section we develop basic notation that we use throughout the paper. We provide this in the following list:
\begin{itemize}
\item $\mct$ is a shape-regular, simplicial triangulation of 
a contractible polytope $\Omega \subset \mathbb{R}^d\ (d=2,3)$.

\item $h_T = {\rm diam}(T)$ for all $T\in \calT_h$ and $h = \max_{T\in \calT_h} h_T$.

\item For an $n$-dimensional simplex $S$ ($n\le d$) and $m\in \{0,\ldots n\}$, denote by $\Delta_m(S)$
the set of $m$-dimensional simplices of $S$.

\item $\pol_r(S)$ denotes the space of polynomials of degree $\le r$ with domain $S$.
Analogous spaces of vector-valued spaces are given in boldface, e.g., $\bpol_r(S) = [\pol_r(S)]^d$.

\item We define 
the following function spaces on $\Omega$:
\begin{align*}
L^2( \Omega)&:=\{w:\Omega \mapsto \bbR:\ \|w\|_{{\it L}^2(\Omega)}:=(\int_{\Omega}\ {|w|^2}\,dx)^{1/2}<\infty \},\\
H^m(\Omega)&:=\{w:\Omega \mapsto \bbR:\ \|w\|_{{\it H}^m(\Omega)}:= (\sum_{|\beta| \leq m} \|{\it D}^\beta w\|_{{\it L}^2( \Omega)}^2)^{1/2}<\infty \},
\end{align*}
and the spaces with boundary conditions:
\begin{align*}
L_0^2( \Omega)&:=\{w\in {\it L}^2( \Omega):\ \int_{\Omega}\ {w}\,dx =0  \},\\
H_0^m(\Omega)&:=\{{w}\in{\it H}^m(\Omega):\ {D}^\beta {w}|_{\partial\Omega}=0,  \forall \beta :\ |\beta| \leq m-1   \}.
\end{align*}
Analogous vector-valued spaces are denoted in boldface, e.g., $\bH^1(\Omega) = [H^1(\Omega)]^d$.

\item For a simplicial triangulation $\calQ_h$ of $\Omega$
 we define the spaces 
of piecewise polynomials
\begin{alignat*}{2}
&\pol_k(\calQ_h) = \prod_{K\in \calQ_h} \pol_k(K),\qquad &&\bpol_k(\calQ_h) = \prod_{K\in \calQ_h} \bpol_k(K),\\
&\mathring{\pol}_k(\calQ_h) = \pol_k(\calQ_h)\cap L^2_0(\Omega),\qquad && \mathring{\bpol}_k^c(\calQ_h) = \bpol_k(\calQ_h)\cap \bH^1_0(\Omega).
\end{alignat*}
Thus $\mathring{\pol}_k(\calQ_h)$ consists of piecewise polynomials of degree $\le k$ with respect 
to the triangulation $\calQ_h$
with mean zero, and $\mathring{\bpol}_k^c(\calQ_h)$ is the space of continuous, piecewise
polynomials of degree $\le k$ with vanishing trace (i.e., the $k$th degree vector-valued Lagrange finite element space).

\item The constant $C$ denotes a generic positive constant, independent of the mesh parameter $h$.
\end{itemize}

\subsection{The Stokes problem}
The Stokes equations defined on a polygonal domain $\Omega \subset \bbR^d$ with Lipschitz continuous boundary $\p\Omega$ is given by the system of equations
\begin{subequations}
\label{eqn:3DProblem}
\begin{alignat}{2}
\label{eqn:3DProblem1}
-\nu \Delta {\bu} + {\nab} p & = {{\bm f}}\qquad &&\text{in }\Omega,\\
\label{eqn:3DProblem2}
\Div \bu & =0 \qquad &&\text{in } \Omega,\\
\label{eqn:3DProblem3}
\bu & = 0\qquad &&\text{on }\p \Omega,
\end{alignat}
\end{subequations}
where the velocity $ \bu = ( u_1, ..., u_d)^\intercal$ and pressure $ p$
are functions of $x = ( x_1, ..., x_d)^\intercal$,
and $ \nab$, $\Delta$ denote the gradient operator and vector Laplacian operator with respect to 
$x$ respectively. In \eqref{eqn:3DProblem1}, $\nu$ is the viscosity.

 The weak formulation for \eqref{eqn:3DProblem} reads:  Find $({\bu}, p)\in   \bH_0^1(\Omega) \times  L_0^2( \Omega)$ such that $\forall ({\bv}, q)\in  \bH_0^1(\Omega)\times  { L}_0^2( \Omega)$ we have
\begin{subequations}\label{Eq:CTS}
\begin{alignat}{2}
\nu\int_{\Omega}{\nabla{\bu}:\nabla{\bv}}- \int_{\Omega}{(\Div \bv) p}&=\int_{\Omega}{{{\bm f}\cdot{\bv}}},\\
 \int_{\Omega}{(\Div{\bu}) q}&=0.
\end{alignat}
\end{subequations}
It is well known that the problem \eqref{Eq:CTS} has a unique solution \cite{GR}.

Let $\mathring{\bV}_h\times \mathring{Y}_h \subset \bH^1_0(\Omega)\times L^2_0(\Omega)$
be a conforming and finite dimensional pair.  Then
the finite element method for \eqref{eqn:3DProblem},
based on the standard velocity-pressure formulation,
seeks $(\bu_h,p_h)\in \mathring{\bV}_h\times \mathring{Y}_h$
such that $\forall (\bv,q)\in \mathring{\bV}_h\times \mathring{Y}_h$ we have
\begin{subequations}\label{Eq:FEM_VP}
\begin{alignat}{2}
\nu\int_{\Omega}{\nabla{\bu_h}:\nabla{\bv}}- \int_{\Omega}{(\Div \bv) p_h}&=\int_{\Omega}{{{\bm f}\cdot{\bv}}},\\
 \int_{\Omega}{(\Div{\bu_h}) q}&=0.
\end{alignat}
\end{subequations}
The discrete problem \eqref{Eq:FEM_VP} is well-posed if and only
if the pair satisfies the inf-sup condition
\[
\sup_{0\neq \bv\in \mathring{\bV}_h} \frac{\int_\Omega (\Div \bv)q}{\|\nabla \bv\|_{L^2(\Omega)}}\ge \beta \|q\|_{L^2(\Omega)},
\]
for some $\beta>0$.

\section{Inf-sup stability on Powell--Sabin Triangulations}\label{PS}

Let $\mct$ be a simplicial, shape-regular triangulation of $\Omega\subset \bbR^2$.
The Powell--Sabin refinement of $\mct$, denoted by $\PST$ is obtained
by the following procedure.  First connect the incenter
of each triangle $T\in \mct$ with its three vertices.
Next, the incenters of each adjacent pair of triangles
are connected with an edge.  If $T$ has a boundary edge,
then the midpoint of the boundary edge is connected to the incenter
of $T$.  Thus, this procedure splits each triangle
in $\mct$ into six subtriangles; cf.~Figure \ref{fig:LocalPS}.

\begin{figure}
\begin{center}
\begin{tikzpicture}[scale=1.5]
\coordinate (C1) at (-1,0);
\coordinate (C2) at (1,0);
\coordinate (C3) at (1.5,1);
\coordinate (C4) at (2.5,-1);

\draw[-](C1)--(C2)--(C3)--(C1);
\draw[-](C1)--(C2)--(C4)--(C1);
\end{tikzpicture}
\begin{tikzpicture}[scale=1.5]
\coordinate (C1) at (-1,0);
\coordinate (C2) at (1,0);
\coordinate (C3) at (1.5,1);
\coordinate (C4) at (2.5,-1);

\coordinate (M13) at (0.25,0.5);
\coordinate (M23) at (1.25,0.5);

\coordinate (M14) at (0.75,-0.5);
\coordinate (M24) at (1.75,-0.5);

\coordinate (P) at (0.8624590165,0);

\draw[-](C1)--(C2)--(C3)--(C1);
\draw[-](C1)--(C2)--(C4)--(C1);
\node[inner sep = 0pt, minimum size=2.pt,fill = black!100,circle] (N1) at (0.79,0.34) {};
\node[inner sep = 0pt, minimum size=2.pt,fill = black!100,circle] (N2) at (0.92,-0.27) {};
\draw[-](N1)--(C1);
\draw[-](N1)--(N2);
\draw[-](N1)--(C2);
\draw[-](N1)--(C3);
\draw[-](N2)--(C1);
\draw[-](N2)--(C2);
\draw[-](N2)--(C4);
\draw[-](N1)--(M23);
\draw[-](N1)--(M13);
\draw[-](N2)--(M14);
\draw[-](N2)--(M24);

\node[inner sep = 0pt,minimum size=3pt,fill=red!100,circle] (n2) at (M14)  {};
\node[inner sep = 0pt,minimum size=3pt,fill=red!100,circle] (n2) at (M24)  {};
\node[inner sep = 0pt,minimum size=3pt,fill=red!100,circle] (n2) at (M13)  {};
\node[inner sep = 0pt,minimum size=3pt,fill=red!100,circle] (n2) at (M23)  {};
\node[inner sep = 0pt,minimum size=3pt,fill=red!100,circle] (n2) at (P)  {};
\end{tikzpicture}
\end{center}
\caption{\label{fig:LocalPS} Two triangles in the mesh $\mct$ (left) and their Powell--Sabin refinement (right).
The singular vertices in the Powell--Sabin refinement are depicted in red.}
\end{figure}
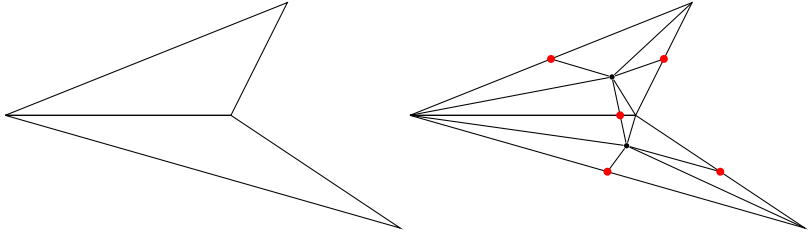

One feature of the Powell-Sabin triangulation is the presence of singular vertices.
\begin{definition}
We say that a vertex in a simplicial triangulation is {\em singular}
if the edges meeting at the vertex fall on exactly two straight lines.
\end{definition}

Let $\calS^I_h$ and $\calS^B_h$ denote the sets of interior and boundary singular vertices
in $\PST$, respectively, and set $\calS_h = \calS^I_h\cup \calS_h^B$.
Note that the cardinalities of the sets $\calS_h^I$ and $\calS^B_h$
correspond to the number of interior and boundary edges in $\mct$, respectively.
For $z\in \calS_h^I$, we denote by $\calT_z\subset \PST$
the set of four triangles that have $z$ as a vertex.
We write $\calT_z = \{K_z^{(1)},K_z^{(2)},K_z^{(3)},K_z^{(4)}\}$ 
with $K_z^{(j)}\in \PST$, labeled such that $K_z^{(j)}$ and $K_z^{(j+1)}$
have a common edge. Likewise, for $z\in \calS_h^B$, we 
set $\calT_z = \{K_z^{(1)},K_z^{(2)}\}\subset \PST$ to denote
the two triangles that have $z$ as a vertex.
We set $n_z = |\calT_z|$, i.e., $n_z = 4$ if $z\in \calS_h^I$
and $n_z = 2$ if $z\in \calS_h^B$.

\subsection{Finite Element Spaces on Powell--Sabin Triangulations}

It is well known that, on singular vertices, the divergence operator acting
on the Lagrange finite element space has ``weak continuity properties''; the precise
meaning of this statement is provided in the following lemma.    Its proof 
can be found in, e.g., \cite{ScottVogelius85,guzman2019scott}.
\begin{lemma}\label{lem:WeakContinuity}
For $z\in \calS_h$, and piecewise smooth function $q$, define
\[
\theta_z(q):=\left\{
\begin{array}{ll}
q|_{K_z^{(1)}}(z)-q|_{K_z^{(2)}}(z)+q|_{K_z^{(3)}}(z)-q|_{K_z^{(4)}}(z) & z\in \calS_h^I,\\
q|_{K_z^{(1)}}(z) - q|_{K_z^{(2)}} & z\in \calS_h^B.
\end{array}
\right.
\]
Then there holds $\theta_z(\Div \bv)=0$ for all $\bv\in \mathring{\bpol}_k^c(\PST)$.
\end{lemma}

Based on Lemma \ref{lem:WeakContinuity},
and on a discrete de Rham complex,
divergence--free
finite element pairs for the Stokes problem have
been constructed and analyzed in \cite{GuzmanLischkeNeilan,zhang2008p1} on Powell--Sabin triangulations.
Here, we focus on the lowest-order case,
where the velocity space is the linear Lagrange space,
and the pressure space is a subspace of piecewise constants
with a weak continuity property at singular vertices. We first define the spaces without boundary conditions and then the ones with boundary conditions. 
\begin{align*}
\bV_h^{\rm PS}:&=\bpol^c_1(\PST),\\
Y_h^{\rm PS}:&=\{q\in \pol_0(\PST):\ \theta_z(q)=0\ \forall z\in \calS_h^I\}.
\end{align*}
We now define an intermediate pressure space
\begin{align*}
\hY_h^{\rm PS}:&=\{q\in \pol_0(\PST):\ \theta_z(q)=0\ \forall z\in \calS_h\}.
\end{align*}
The spaces with boundary conditions are 
\begin{align*}
\mbV_h^{\rm PS}:&= \bV_h^{\rm PS} \cap \bH_0^1(\Omega),\\
\mY_h^{\rm PS}:&= \hY_h^{\rm PS} \cap L_0^2(\Omega).
\end{align*}

\subsection{Stability of $\mbV_h^{\rm PS}\times \mY_h^{\rm PS}$}
The stability of the pair $\mbV_h^{\rm PS}\times \mY_h^{\rm PS}$ is implicitly shown in \cite{BoffiGuzmanNeilan20}. Here we give more details. 
%
\begin{theorem}\label{lem:InfSup}
There holds $\Div \mbV_h^{\rm PS}\subseteq \mY_h^{\rm PS}$.  Moreover, the pair $\mbV_h^{\rm PS}\times \mY_h^{\rm PS}$ is an inf-sup stable pair for the Stokes problem, i.e.,
\[
\sup_{0\neq \bv\in \mbV_h^{\rm PS}} \frac{\int_{\Omega} (\Div \bv)q}{\|\nab \bv\|_{L^2(\Omega)}} \ge \beta \|q\|_{L^2(\Omega)}\quad \forall q\in \mY_h^{\rm PS},
\]
for some $\beta>0$ independent of the mesh parameter $h$.
\end{theorem}
\begin{proof}
For fixed $q\in \mY_h^{\rm PS}$, let $\bw\in \bH^1_0(\Omega)$ satisfy $\Div \bw = q$ and $\|\bw\|_{H^1(\Omega)}\le C \|q\|_{L^2(\Omega)}$.
\cite[Theorem 3.4]{BoffiGuzmanNeilan20}
gives the existence of an operator $\bPi_h$ such that $\bPi_h \bw\in \bV_h^{\rm PS}$, $\Div \bPi_h \bw = \Div \bw = q$, and
\[
\|\bw-\bPi_h \bw\|_{L^2(T)}\le C h_T \big[\|\bw\|_{H^1(\omega(T))} + \|\Div \bw\|_{L^2(T)}\big]\qquad \forall T\in \calT_h,
\]
where $\omega(T) = \mathop{\cup_{T'\in \calT_h}}_{\bar T'\cap \bar T \neq \emptyset} T'$.
Standard arguments then show $\|\nab \bPi_h \bw\|_{L^2(\Omega)}\le C\big(\|\bw\|_{H^1(\Omega)}+ \|\Div \bw\|_{L^2(\Omega)}\big)\le C\|q\|_{L^2(\Omega)}$.
Setting $\bv = \bPi_h \bw$, we have
\begin{align*}
\int_\Omega (\Div \bv) q\, dx = \int_\Omega (\Div \bPi_h \bw) q\, dx = \|q\|_{L^2(\Omega)}^2 \ge C^{-1} \|q\|_{L^2(\Omega)} \|\nab \bv\|_{L^2(\Omega)}.
\end{align*}
This yields the inf-sup stability result with $\beta = C^{-1}$.
\end{proof}


\section{Inf-sup stability on  Worsey-Farin splits}\label{WF}

Let $\mct$ be a simplicial triangulation of a polyhedral domain $\Omega\subset \bbR^3$.
The Worsey-Farin triangulation $\WFT$ is obtained by splitting each tetrahedron into twelve
sub-tetrahedra by the following procedure.  Similar to the Powell--Sabin case, for each $T\in \mct$,
we connect the incenter of $T$ to its vertices.  Next, the incenters of neighboring
pairs of tetrahedra are connected with an edge.  This creates a face split point (a vertex)
on each face of $T$.  If $T$ has a boundary face, then we connect the incenter of $T$
to the barycenter of the face by a line.  Finally, the face split points are connected to the vertices
of the face. For each $T\in \calT_h$, we denote by $T^{\rm WF}$ the triangulation resulting 
from local Worsey-Farin refinement of $T$, i.e.,
\[
T^{\rm WF} = \{K\in \WFT:\ K\subset \bar T\}.
\]

\begin{definition}
An edge in a 3D simplicial triangulation is called {\em singular}
if the faces meeting at the edge fall on exactly two planes.
\end{definition}
By construction, the Worsey-Farin triangulation
contains many singular edges; for each
face in the unrefined triangulation $\mct$,
there are three associated singular edges in $\WFT$.

Let $\mathcal{E}^{\mathcal{S}}_h$ denote the set of singular edges
in $\WFT$, and let $\calE^{\calS,I}_h$ and $\calE_h^{\calS,B}$ denote
the sets of interior and boundary singular edges, respectively.
For each $e\in \calE^{\calS}_h$, let 
$\calT_e = \{K_e^{(1)},\ldots,K_e^{(n_e)}\}$ denote
the set of tetrahedra that have $e$ as an edge.
Here, $n_e=4$ if $e$ is an interior edge,
and $n_e = 2$ if $e$ is a boundary edge.
We assume the tetrahedra are labeled such that
$K_e^{(j)}$ and $K_e^{(j+1)}$ share a common face.

\subsection{Finite Element Spaces on Worsey--Farin Triangulations}
Similar to the two-dimensional case, the 
divergence operator acting on the Lagrange finite element
space has weak continuity properties on singular edges (cf.~\cite{GLN20}).
%
\begin{lemma}
For $e\in \calE^{\calS}_h$, and a piecewise smooth function $q$, define
\[
\theta_e(q) = \left\{
\begin{array}{ll}
q_e^{(1)}|_e - q_e^{(2)}|_e +q_e^{(3)}|_e -q_e^{(4)}|_e  & e\in \calE^{\calS,I}_h,\\
q_e^{(1)}|_e - q_e^{(2)}|_e & e\in \calE^{\calS,B}_h,
\end{array}
\right.
\]
where $q_e^{(j)} = q|_{K_e^{(j)}}$.  Then
there holds $\theta_e(\Div \bv)=0$ for all $\bv\in \mathring{\bpol}_k^c(\WFT)$.
\end{lemma}

Analogous to the Powell--Sabin case, we define the finite element spaces to discretize the Stokes problem on Worsey-Farin splits. We first define the spaces without boundary conditions
\begin{align*}
\bV_h^{\rm WF}& = \bpol_1^c(\WFT),\\
Y_h^{\rm WF}& = \{q\in \mathring{\pol}_0(\WFT):\ \theta_e(q)=0\ \forall e\in \calE^{\calS,I}_h\}.
\end{align*}
Then, we define an intermediate pressure space
\begin{align*}
\hY_h^{\rm WF}& = \{q\in \mathring{\pol}_0(\WFT):\ \theta_e(q)=0\ \forall e\in \calE^{\calS}_h\}.
\end{align*}
We now define the spaces with boundary conditions
\begin{align*}
\mbV_h^{\rm WF}& =\bV_h^{\rm WF} \cap \bH_0^1(\Omega),\\
\mY_h^{\rm WF}& = \hY_h^{\rm WF} \cap  L_0^2(\Omega).
\end{align*}

\subsection{Stability of $\mbV_h^{\rm WF}\times \mY_h^{\rm WF}$}
In this section, we show  the pair $\mbV_h^{\rm WF}\times \mY_h^{\rm WF}$ is inf-sup stable.
First we introduce some notation.

Let $T\in \mct$, and let $T^{\rm A}$ denote the local triangulation of $T$
consisting of four tetrahedra, obtained by connecting the vertices of $T$
with its incenter, i.e., $T^{\rm A}$ denotes the Alfeld split of $T$.
For a face $F\subset T$, denote by $F^{\rm CT}$ 
the set of three triangles formed from $F$ by the Worsey-Farin refinement, i.e.,
$F^{\rm CT}$ is the Clough-Tocher refinement of $F$.
We denote by $\Delta_1^I(F^{\rm CT})$ the set of three interior edges in $F^{\rm CT}$,
and let $e_F\in \Delta_1^I(F^{\rm CT})$ denote an arbitrary, fixed interior edge of $F^{\rm CT}$.

\begin{definition}
Consider the triangulation $F^{CT}$
of a face $F\in \Delta_2(T)$, and
let the three triangles of $F^{\rm CT}$ be labeled $Q_1,Q_2,Q_3$.
Let $e = \p Q_1\cap \p Q_2$ be an internal edge, let $\bt$ be the unit vector
tangent to $e$ pointing away from the split point $m_F$, and let $\bs$ be the unit vector 
orthogonal to $\bt$ such that $\bs\times \bn_F = \bt$.
Then the jump of a piecewise smooth function $p$ across $e$ is defined as
\[
\jump{p}_e = \big(p|_{Q_1} - p|_{Q_2}\big)\bs.
\]
\end{definition}

We now state the degrees of freedom
for the spaces  $\mbV_h^{\rm WF}$ and  $\mY_h^{\rm WF}$.
The proofs of the following two lemmas are given in \cite[Lemmas 5.11--5.12]{GLN20}.
\begin{lemma}
\label{lem:VelocityDOFs}
A function $\bv\in \bV_h^{\rm WF}$ is uniquely determined by the values 
\begin{alignat*}{2}
&\bv(a)\qquad &&\forall a\in \Delta_0(T),\\
&\int_F (\bv\cdot \bn_F)\qquad &&\forall F\in \Delta_2(T),\\
&\int_e \jump{\Div \bv}_e 
\qquad &&\forall e\in \Delta_1^I(F^{\rm CT})\backslash \{e_F\},\ \forall F\in \Delta_2(T),\\
&\int_T (\Div \bv) p\qquad &&\forall p\in \mathring{\mathcal{V}}_0(T):=\pol_0(T^{\rm A})\cap L^2_0(T).
\end{alignat*}
for each $T\in \calT_h$.
\end{lemma}

\begin{lemma}\label{lem:PressureDOFs}
A function $q\in Y_h^{\rm WF}$ is uniquely determined
by 
\begin{alignat*}{2}
&\int_e \jump{q}_e \qquad &&\forall e\in \Delta_1^I(F^{\rm CT})\backslash \{e_F\},\ \forall F\in \Delta_2(T),\\
%
%
&\int_T q p \qquad && \forall p\in \pol_0(T^{\rm A}).
\end{alignat*}
for all $T\in \calT_h$. 
\end{lemma}
If we restrict ourselves to $\hY_h^{\rm WF}$ then we only take interior faces $F$ in the first set of degrees of freedom.

\begin{proposition}\label{prop:WFScaling}
Let $\bv\in \bV_h^{\rm WF}$ and $T\in \calT_h$. For $m=0, 1$, 
there holds
\begin{align*}
|\bv|_{H^m(T)}
&\le C h_T^{-1-2m}\Big(h_T^4 \sum_{ a\in \Delta_0( T)}|\bv(a)|^2
+ \sum_{ F\in \Delta_2( T)} \Big|\int_{ F}  \bv \cdot  \bn_{ F}\Big|^2
%
 +  h_T^3\|\Div \bv\|_{L^2(T)}^2\Big).
\end{align*}
\end{proposition}
\begin{proof}
Let $\hat T$ be the reference tetrahedron, 
and let $F_T:\hat T\to T$ be an affine bijection with $F_T (\hat x) = A_T \hat x + b_T$
with $A_T\in \mathbb{R}^{3\times 3}$ and $b_T\in \mathbb{R}^3$.
We define $\hat \bv:\hat T\to \mathbb{R}^3$ via the Piola transform
\[
 \bv( x) = \frac{A_T \hat \bv(\hat x)}{\det(A_T)},\qquad x = F_T(\hat x).
\]

Let $\hat T^{\rm WF}$ be the split of $\hat T$ induced by $T^{\rm WF}$ and the mapping $F_T^{-1}$, i.e.,
\[
\hat T^{\rm WF} = \{F_T^{-1}(K):\ K\in T^{\rm WF}\}.
\]
Then $\hat \bv$ is a continuous piecewise linear polynomial with respect to
$\hat T^{\rm WF}$, and therefore by equivalence of norms, and Lemma \ref{lem:VelocityDOFs},
\begin{align*}
|\hat \bv|_{H^m(\hat T)}^2
&\le C \Big(\sum_{\hat a\in \Delta_0(\hat T)} |\hat \bv(\hat a)|^2
+ \sum_{\hat F\in \Delta_2(\hat T)} \Big|\int_{\hat F} \hat \bv \cdot \hat \bn_{\hat F}\Big|^2\\
&\qquad + \sum_{\hat F\in \Delta_2(\hat T)} \sum_{\hat e\in \Delta^I_1(\hat F^{\rm CT})\backslash \{\hat e_{\hat F}\}} \Big| \int_{\hat e} \jump{\widehat{\Div}\hat \bv}_{\hat e}\Big|^2 + \mathop{\sup_{\hat p\in \mathring{\mathcal{V}}_0(\hat T)}}_{\|\hat p\|_{L^2(\hat T)}=1} \Big| \int_{\hat T} (\widehat \Div \hat \bv)\hat p\Big|^2\Big).
\end{align*}
By well-known properties of the Piola transform, we have
\begin{align*}
\Div \bv(x) = \frac1{\det(A_T)} \widehat \Div \hat \bv(\hat x),\qquad \int_F \bv\cdot \bn_F = \int_{\hat F} \hat \bv \cdot \hat \bn_{\hat F}.
\end{align*}
Thus, we have
\begin{align*}
|\hat \bv|_{H^m(\hat T)}^2
&\le C \Big(\sum_{ a\in \Delta_0( T)} |\det(A_T) A_T^{-1} \bv(a)|^2
+ \sum_{ F\in \Delta_2( T)} \Big|\int_{ F}  \bv \cdot  \bn_{ F}\Big|^2\\
&\qquad + |\det(A_T)|^2 \sum_{ F\in \Delta_2( T)} \sum_{ e\in \Delta_1( F)\backslash \{ e_{ F}\}} \Big| \frac{|\hat e|}{|e|} \int_e \jump{{\Div} \bv}_{ e}\Big|^2
 + \mathop{\sup_{\hat p\in \mathring{\mathcal{V}}_0(\hat T)}}_{\|\hat p\|_{L^2(\hat T)}=1} \Big| \int_{\hat T} (\widehat \Div \hat \bv)\hat p\Big|^2\Big).
\end{align*}

Next, for $\hat p\in \mathring{\mathcal{V}}_0(\hat T)$ with $\|\hat p\|_{L^2(\hat T)}=1$, let $p:T\to \mathbb{R}$
be given by $p(x) = \hat p(\hat x)$.  Then $p\in \mathring{\mathcal{V}}_0( T)$,  $\|p\|_{L^2(T)} = \sqrt{6|T|}$, and
\[
\int_{\hat T} (\widehat \Div \hat \bv)\hat p = \int_T  (\Div \bv)p.
\]
We conclude 
\begin{align*}
\mathop{\sup_{\hat p\in \mathring{\mathcal{V}}_0(\hat T)}}_{\|\hat p\|_{L^2(\hat T)}=1} \Big| \int_{\hat T} (\widehat \Div \hat \bv)\hat p\Big|^2
&\le \mathop{\sup_{ p\in \mathring{\mathcal{V}}_0( T)}}_{\| p\|_{L^2( T)}=\sqrt{6|T|}} \Big| \int_T ( \Div  \bv) p\Big|^2\le C h_T^3\|\Div \bv\|_{L^2(T)}^2.
\end{align*}

Finally, we use $\|A_T^{-1}\|\le Ch_T^{-1}$ and $|\det(A_T)| = 6|T|\le Ch_T^3$ to get
\begin{align*}
|\hat \bv|_{H^m(\hat T)}^2
&\le C \Big(h_T^4 \sum_{ a\in \Delta_0( T)}|\bv(a)|^2
+ \sum_{ F\in \Delta_2( T)} \Big|\int_{ F}  \bv \cdot  \bn_{ F}\Big|^2\\
&\qquad + h_T^4\sum_{ F\in \Delta_2( T)} \sum_{ e\in \Delta^I_1( F^{\rm CT})\backslash \{ e_{ F}\}} \Big| \int_e \jump{{\Div} \bv}_{ e}\Big|^2
 +  h_T^3\|\Div \bv\|_{L^2(T)}^2\Big),
\end{align*}
and therefore
\begin{align*}
|\bv|_{H^m(\hat T)}^2 \le C h_T^{-1-2m}|\hat \bv|_{H^m(\hat T)}^2
&\le C h_T^{-1-2m}\Big(h_T^4 \sum_{ a\in \Delta_0( T)}|\bv(a)|^2
+ \sum_{ F\in \Delta_2( T)} \Big|\int_{ F}  \bv \cdot  \bn_{ F}\Big|^2\\
&\qquad + h_T^4\sum_{ F\in \Delta_2( T)} \sum_{ e\in \Delta_1( F)\backslash \{ e_{ F}\}} \Big| \int_e \jump{{\Div} \bv}_{ e}\Big|^2
 +  h_T^3\|\Div \bv\|_{L^2(T)}^2\Big)\\
 &\le C h_T^{-1-2m}\Big(h_T^4 \sum_{ a\in \Delta_0( T)}|\bv(a)|^2
+ \sum_{ F\in \Delta_2( T)} \Big|\int_{ F}  \bv \cdot  \bn_{ F}\Big|^2
%
 +  h_T^3\|\Div \bv\|_{L^2(T)}^2\Big),
 \end{align*}
 where the last inequality comes from standard trace and inverse inequalities.
\end{proof}

\begin{theorem}
The pair $\mbV_h^{\rm WF}\times \mY_h^{\rm WF}$ is inf-sup stable.
\end{theorem}
\begin{proof}

Fix a $q\in \mY_h^{\rm WF}$, and let $\bw\in \bH^1_0(\Omega)$ satisfy
$\Div \bw = q$ and $\|\nab \bw\|_{L^2(\Omega)}\le C\|q\|_{L^2(\Omega)}$.
Let $\bw_h\in  \mathring{\bpol}_1(\calT_h)$ be the Scott-Zhang interpolant of $\bw$
with respect to $\calT_h$.
Define $\bv\in \mathring{\bpol}_1(\WFT)$ such that
\begin{alignat*}{2}
&\bv(a) = \bw_h(a)\qquad &&\forall a\in \Delta_0(T),\\
&\int_F (\bv\cdot \bn_F) = \int_F (\bw\cdot \bn_F)\qquad &&\forall F\in \Delta_2(T),\\
&\int_e \jump{\Div \bv}_e  = \int_e \jump{q}_e
\qquad &&\forall e\in \Delta_1^I(F^{\rm CT})\backslash \{e_F\},\ \forall F\in \Delta_2(T),\\
&\int_T (\Div \bv) p = \int_T qp\qquad &&\forall p\in \mathring{\mathcal{V}}_0(T).
\end{alignat*}

Noting $(\Div \bv-q)\in {\bm Y}_h^{\rm WF}$,
and
\begin{alignat*}{2}
&\int_e \jump{\Div \bv-q}_e=0 \qquad &&\forall e\in \Delta_1^I(F^{\rm CT})\backslash \{e_F\},\ \forall F\in \Delta_2(T),\\
%
%
&\int_T (\Div \bv-q)  p \qquad && \forall p\in \pol_0(T^{\rm A})
\end{alignat*}
for all $T\in \calT_h$ (by the divergence theorem), we conclude $\Div \bv=q$ by Lemma \ref{lem:PressureDOFs}.

We apply Proposition \ref{prop:WFScaling} to $(\bv-\bw_h)$ with $m=1$:
\begin{align*}
\|\nab (\bv-\bw_h)\|_{L^2(T)}^2
&\le C h_T^{-3}\Big(h_T^4 \sum_{ a\in \Delta_0( T)}|(\bv-\bw_h)(a)|^2
+ \sum_{ F\in \Delta_2( T)} \Big|\int_{ F}  (\bv-\bw_h) \cdot  \bn_{ F}\Big|^2\\
&\qquad 
 +  h_T^3\|\Div (\bv-\bw_h)\|_{L^2(T)}^2\Big)\\
&= C h_T^{-3}\Big( \sum_{ F\in \Delta_2( T)} \Big|\int_{ F}  (\bw-\bw_h) \cdot  \bn_{ F}\Big|^2
%
 +  h_T^3\|q-\Div \bw_h\|_{L^2(T)}^2\Big) \\
 &\le C h_T^{-3}\Big( h_T^2 \sum_{ F\in \Delta_2( T)} \|\bw-\bw_h\|_{L^2(F)}^2
%
 +  h_T^3\|q-\Div \bw_h\|_{L^2(T)}^2\Big) \\
 &\le C\Big( \|q\|_{L^2(T)}^2 + h_T^{-2}\|\bw-\bw_h\|_{L^2(T)}^2 + \|\nab(\bw-\bw_h)\|_{L^2(T)}^2 + \|\nab \bw_h\|_{L^2(T)}^2\Big).
\end{align*}

We then sum over $T\in \mct$ and apply stability and approximation properties
of the Scott-Zhang interpolant to conclude $\|\nab \bv\|_{L^2(\Omega)}\le C\|q\|_{L^2(\Omega)}$.

\end{proof}


\section{Implementation Aspects for  Powell-Sabin splits}\label{implPS}
The only tricky part to implement these finite elements is the pressure spaces since they have non-standard constraints in their definitions. 
In this section and the subsequent one, we give details to form the algebraic system for the Stokes problem.

\subsection{A basis for $\hY_h^{\rm PS}$ and the construction of the algebraic system}
Recall that the space $\hY_h^{\rm PS}$ consists
of piecewise constant functions satisfying a weak continuity property
at vertices.  Clearly, this is a non-standard space, and in particular
the space is not explicitly found in current finite element software packages.
Nonetheless, in this section, we identify a basis of the space $\mY_h^{\rm PS}$,
and as a byproduct show that the weak continuity property $\theta_z(q)=0$
can be imposed purely at the algebraic level.

As a first step, we note that, by definition of the Powell--Sabin triangulation,
\[
\PST = \{K_z^{(j)}:\ K_z^{(j)}\in \calT_z,\ z\in \calS_h\}.
\]
With this (implicit) labeling of the triangles in $\PST$, we
can write the canonical basis of $\pol_0(\PST)$
as the set $\{\varphi_z^{(j)}\}\subset \pol_0(\PST)$ with
\[
\varphi_z^{(j)}|_{K_v^{(i)}} = \delta_{v,z} \delta_{i,j}\qquad \forall z,v\in \calS_h,\ i=1,\ldots,n_v,\ j=1,\ldots,n_z.
\]

The next proposition shows that a basis of $\mY_h^{\rm PS}$ is easily extracted
from the basis of $\pol_0(\PST)$ (see Figure \ref{fig:Basis}).
\begin{proposition}\label{prop:PSBasis}
For each $z\in \calS_h$ and $j\in \{2,\ldots,n_z\}$, define
\[
\psi_z^{(j)} = \varphi_z^{(j)}+(-1)^j \varphi_z^{(1)}.
\]
Then $\{\psi_z^{(j)}:\ z\in \calS_h,\ j=2,\ldots,n_z\}$ forms a basis of $\hY_h^{\rm PS}$.
\end{proposition}
\begin{proof}
Note that the number of functions $\{\psi_z^{(j)}\}$ given
is $ \sum_{z\in \calS_h} (n_z-1)$, and
\[
\dim Y_h^{PS} = \dim \pol_0(\PST) - |\calS_h| = \sum_{z\in \calS_h} n_z - |\calS_h| =  \sum_{z\in \calS_h} (n_z-1).
\]
Because $\psi_z^{(j)}\in \mY_h^{\rm PS}$,
and they are clearly linear independent, we conclude that $\{\psi_z^{(j)}\}$
form a basis of $\mY_h^{\rm PS}$.
\end{proof}

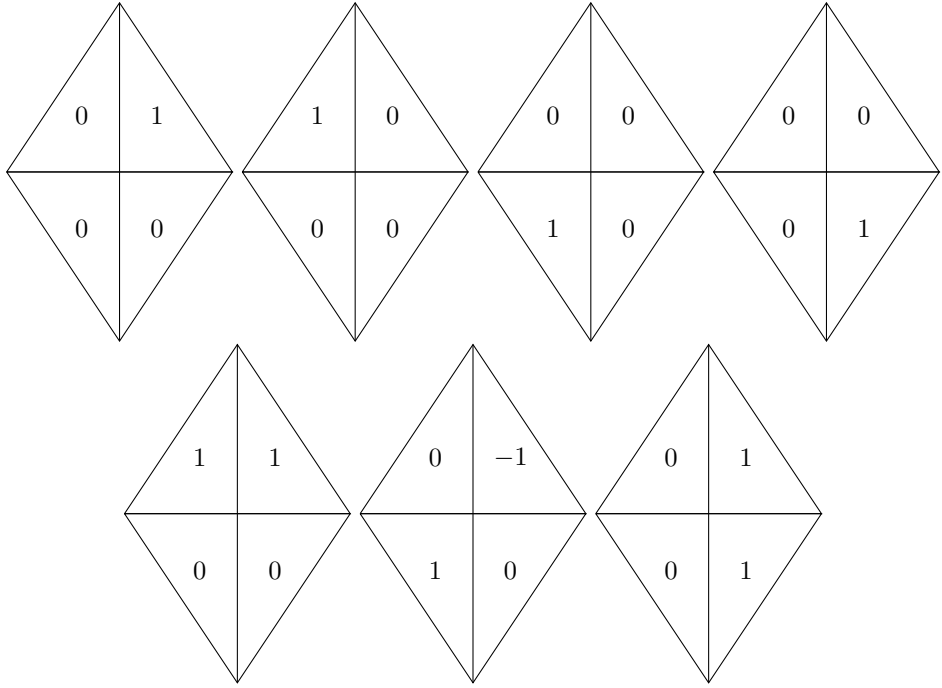
\begin{figure}
\begin{center}
\begin{tikzpicture}[scale=1.5]
\coordinate (C1) at (-1,0);
\coordinate (C2) at (1,0);
\coordinate (C3) at (0,1.5);
\coordinate (C4) at (0,-1.5);
\draw[-](C1)--(C2)--(C3)--(C1);
\draw[-](C1)--(C2)--(C4)--(C1);
\draw[-](C3)--(C4);

\draw (0.33333,0.5) node {$1$};
\draw (-0.33333,0.5) node {$0$};
\draw (0.33333,-0.5) node {$0$};
\draw (-0.33333,-0.5) node {$0$};
\end{tikzpicture}
\begin{tikzpicture}[scale=1.5]
\coordinate (C1) at (-1,0);
\coordinate (C2) at (1,0);
\coordinate (C3) at (0,1.5);
\coordinate (C4) at (0,-1.5);
\draw[-](C1)--(C2)--(C3)--(C1);
\draw[-](C1)--(C2)--(C4)--(C1);
\draw[-](C3)--(C4);

\draw (0.33333,0.5) node {$0$};
\draw (-0.33333,0.5) node {$1$};
\draw (0.33333,-0.5) node {$0$};
\draw (-0.33333,-0.5) node {$0$};
\end{tikzpicture}
\begin{tikzpicture}[scale=1.5]
\coordinate (C1) at (-1,0);
\coordinate (C2) at (1,0);
\coordinate (C3) at (0,1.5);
\coordinate (C4) at (0,-1.5);
\draw[-](C1)--(C2)--(C3)--(C1);
\draw[-](C1)--(C2)--(C4)--(C1);
\draw[-](C3)--(C4);
\draw (0.33333,0.5) node {$0$};
\draw (-0.33333,0.5) node {$0$};
\draw (0.33333,-0.5) node {$0$};
\draw (-0.33333,-0.5) node {$1$};
\end{tikzpicture}
\begin{tikzpicture}[scale=1.5]
\coordinate (C1) at (-1,0);
\coordinate (C2) at (1,0);
\coordinate (C3) at (0,1.5);
\coordinate (C4) at (0,-1.5);
\draw[-](C1)--(C2)--(C3)--(C1);
\draw[-](C1)--(C2)--(C4)--(C1);
\draw[-](C3)--(C4);

\draw (0.33333,0.5) node {$0$};
\draw (-0.33333,0.5) node {$0$};
\draw (0.33333,-0.5) node {$1$};
\draw (-0.33333,-0.5) node {$0$};
\end{tikzpicture}\\
\begin{tikzpicture}[scale=1.5]
\coordinate (C1) at (-1,0);
\coordinate (C2) at (1,0);
\coordinate (C3) at (0,1.5);
\coordinate (C4) at (0,-1.5);
\draw[-](C1)--(C2)--(C3)--(C1);
\draw[-](C1)--(C2)--(C4)--(C1);
\draw[-](C3)--(C4);

\draw (0.33333,0.5) node {$1$};
\draw (-0.33333,0.5) node {$1$};
\draw (0.33333,-0.5) node {$0$};
\draw (-0.33333,-0.5) node {$0$};
\end{tikzpicture}
\begin{tikzpicture}[scale=1.5]
\coordinate (C1) at (-1,0);
\coordinate (C2) at (1,0);
\coordinate (C3) at (0,1.5);
\coordinate (C4) at (0,-1.5);
\draw[-](C1)--(C2)--(C3)--(C1);
\draw[-](C1)--(C2)--(C4)--(C1);
\draw[-](C3)--(C4);

\draw (0.33333,0.5) node {$-1$};
\draw (-0.33333,0.5) node {$0$};
\draw (0.33333,-0.5) node {$0$};
\draw (-0.33333,-0.5) node {$1$};
\end{tikzpicture}
\begin{tikzpicture}[scale=1.5]
\coordinate (C1) at (-1,0);
\coordinate (C2) at (1,0);
\coordinate (C3) at (0,1.5);
\coordinate (C4) at (0,-1.5);
\draw[-](C1)--(C2)--(C3)--(C1);
\draw[-](C1)--(C2)--(C4)--(C1);
\draw[-](C3)--(C4);

\draw (0.33333,0.5) node {$1$};
\draw (-0.33333,0.5) node {$0$};
\draw (0.33333,-0.5) node {$1$};
\draw (-0.33333,-0.5) node {$0$};
\end{tikzpicture}

\end{center}
\caption{\label{fig:Basis}Local mesh $\calT_z$ with $z\in \calS_h^I$.
Top row: Values of canonical basis functions of piecewise constants $\{\varphi_z^{(j)}\}_{j=1}^{n_z}$.
Bottom row: Values of basis functions of piecewise constants
with weak continuity constraint $\{\psi_z^{(j)}\}_{j=2}^{n_z}$.
}
\end{figure}

We now explain how 
Proposition \ref{prop:PSBasis} provides a simple way to construct
the stiffness matrix for the Stokes problem
using the $\mbV_h^{\rm PS}\times \mY_h^{\rm PS}$
pair.  To explain
the procedure, we require some notation.

Let $A$ be the matrix associated with the bilinear form
\[
(\bv,\bw)\to \int_{\Omega} \nu \nab \bv:\nab \bw\, dx\quad \text{over }\bv,\bw\in \mbV_h^{\rm PS},
\]
and let $\tilde B$ is the matrix associated with the bilinear form
\[
(\bv,q)\to -\int_{\Omega} (\Div \bv)q\, dx\quad \text{over } \bv\in \mbV_h^{\rm PS},\ q\in \pol_0(\PST).
\]
The stiffness matrix for the Stokes problem based on the (unstable) $\mbV_h^{\rm PS}\times \pol_0(\PST)$
pair is given by
\begin{align*}
\begin{pmatrix}
A & \tilde B\\
\tilde B^\intercal & 0
\end{pmatrix}.
\end{align*}
We emphasize that this system can be easily constructed using standard finite element software packages.

Let $\{\bphi^{(i)}\}_{k=1}^N$ denote a basis of $\mathring{\bV}_h^{\rm PS}$
with $N = \dim \mathring{\bV}_h^{\rm PS}$ so that
\[
A_{i,j} = \nu \int_\Omega \nab \bphi^{(j)}: \nab \bphi^{(i)}\, dx.
\]
Let $M = \dim \pol_0(\PST)$, the number of triangles
in $\PST$, and introduce the local-to-global label mapping
$\sigma:\calS\times \{1,\ldots,n_z\}\to \{1,2,\ldots,M\}$ such that
\[
\tilde B_{i,\sigma(z,j)} = -\int_\Omega (\Div \bphi^{(i)}) \varphi_z^{(j)}\, dx.
\]
Then by Proposition \ref{prop:PSBasis}, we have for $z\in \calS_h$ and $j=2,\ldots,n_z$,
\begin{align*}
 -\int_\Omega (\Div \bphi^{(i)}) \psi_z^{(j)}\, dx 
 &= 
  -\int_\Omega (\Div \bphi^{(i)}) \varphi_z^{(j)}\, dx -(-1)^{j} \int_\Omega (\Div \bphi^{(i)}) \varphi_z^{(1)}\, dx\\
  & = \tilde B_{i,\sigma(z,j)} + (-1)^j \tilde B_{i,\sigma(z,1)}.
 \end{align*}
This identity leads to the following algorithm.

\begin{mathframed}
{\bf Algorithm}
\begin{enumerate}
\item Construct Powell--Sabin triangulation $\PST$

\item Construct $\tilde B\in \bbR^{N\times M}$ based on the $\mbV_h^{\rm PS}\times \pol_0(\PST)$ pair.
\item Set $B = \tilde B$.

\item For each $z\in \calS_h$ and for each $j\in \{2,\ldots,n_z\}$, do the elementary column operation
\[
B_{:,\sigma(z,j)} =B_{:,\sigma(z,j)} +(-1)^j B_{:,\sigma(z,1)}.
\]

\item Delete column $B_{:,\sigma(z,1)}$ for each $z\in \calS_h$. \\

\end{enumerate}
\end{mathframed}

The stiffness matrix for the Stokes problem
based on the $\mbV_h^{\rm PS}\times \hY_h^{\rm PS}$ pair is then given by
\begin{align}\label{eqn:SaddleSystem}
\begin{pmatrix}
A &  B\\
 B^\intercal & 0
\end{pmatrix}.
\end{align}
%
\section{Implemenation aspects for Worsey Farin Splits}\label{implWF}

\subsection{A basis for $\hY_h^{\rm WF}$ and the construction of the algebraic system}

Notice that the collection of local triangulations $\calT_e$ (with $e\in \calE^{\calS}_h$)
do not form a disjoint partition of the global triangulation $\WFT$.
In particular, there exists $K\in \WFT$ such that $K\in \calT_e$ and $K\in \calT_s$
with $e,s\in \calE^{\calS}_h$ and $e\neq s$.  As a result, the methodology used in the previous section
for Powell--Sabin meshes is not directly applicable.

Instead, we consider a geometric decomposition of the mesh based on the face split
points in $\WFT$.  To this end, we denote by $\calS^I_h$ and $\calS_h^B$ the sets of interior and boundary face split points, respectively,
and set $\calS_h = \calS^I_h\cup \calS_h^B$.
For $z\in \calS_h$, 
let $\calT_z:=\{K_z^{(1)},\ldots,K_z^{(n_z)}\}$ denote the set of tetrahedra
in $\WFT$ that have $z$ as a vertex.  Here, $n_z = 6$ if $z$ is an interior vertex,
and $n_z=3$ if $z$ is a boundary vertex.  For an interior face split point $z$, 
we assume  the simplices
in $\calT_z$ are labeled such that 
\[
K_z^{(1)},K_z^{(2)},K_z^{(3)}\subset T^{(1)},\qquad K_z^{(4)},K_z^{(5)},K_z^{(6)}\subset T^{(2)}
\]
for some $T^{(1)},T^{(2)}\in \mct$, and that $K_z^{(j)}$ and $K_z^{(j+3)}$ share
a common face for $j\in \{1,2,3\}$.  For a boundary split point $z$,
the set $\calT_z = \{K_z^{(1)},K_z^{(2)},K_z^{(3)}\}$ is labeled arbitrarily.

We clearly have
\begin{equation}\label{eqn:3DDecomposition}
\WFT = \{K_z^{(j)}:\ z\in \calS_h, j=1,\ldots,n_z\},
\end{equation}
and the map $(z,j)\to K_z^{(j)}$ is injective.  Furthermore, each local partition $\calT_z$
contains three singular edges.
\begin{proposition}\label{prop:WFDim}
There holds
\[
\dim \hY_h^{\rm WF} = 4 |\calS_h^I|+|\calS_h^B|.
\]
\end{proposition}

\begin{proof}
By Proposition \ref{lem:PressureDOFs}, we have
\[
\dim \hY_h^{\rm WF} = 4 |\calT_h| + 2 |\calF^I_h|,
\]
where $|\calF^I_h|$ is the number of interior faces in $\calT_h$.
From \eqref{eqn:3DDecomposition}, we have
\[
12 |\calT_h| = |\WFT|  = 6|\calS_h^I| + 3|\calS_h^B|,
\]
and by construction of the Worsey-Farin split, there holds
\[
|\calF_h| = |\calS^I_h|.
\]
Therefore,
\begin{align*}
\dim \hY_h^{\rm WF} 
&= 4 |\calT_h| + 2 |\calF^I_h| = \frac13\big(6|\calS_h^I| + 3|\calS_h^B|\big)+2|\calS_h^I|= 4 |\calS_h^I| + |\calS_h^B|.
\end{align*}
\end{proof}

For an interior split point $z$, and for a piecewise constant function $q$ on $\calT_z$, the three constraints
$\theta_e(q)=0$ read
\begin{align*}
q_1 - q_2+q_5-q_4 &=0,\\
q_2 -q_3+q_6-q_5 & = 0,\\
q_3 - q_1+q_4 -q_6 & =0,
\end{align*}
where $q_j = q|_{K^{(j)}_z}$

We write this as a $3\times 6$ linear system
\begin{align*}
C\vec{q}:=\begin{pmatrix}
1 & -1 & 0 & -1 & 1 & 0\\ 
0 & 1 & -1 & 0 & -1 & 1\\
-1 &0 & 1 & 1 & 0 & -1
\end{pmatrix}
\begin{pmatrix}
q_1\\
q_2\\
q_3\\
q_4\\
q_5\\
q_6
\end{pmatrix}=0.
\end{align*}
We clearly see that this matrix has rank $2$ (e.g., adding the first and third rows gets the negation of the second row).
We find that 
the nullspace of $C$ is given by
\begin{align*}
{\rm null}(C)= {\rm span}\left\{
\begin{pmatrix}
1\\
1\\
1\\
0\\
0\\
0\end{pmatrix},
\begin{pmatrix}
1\\
0\\
0\\
1\\
0\\
0
\end{pmatrix},
\begin{pmatrix}
0\\
1\\
0\\
0\\
1\\
0
\end{pmatrix},
\begin{pmatrix}
-1\\
-1\\
0\\
0\\
0\\
1
\end{pmatrix}
\right\}.
\end{align*}

These four vectors implicitly give us a basis for $\hY_h^{\rm WF}$.
In particular, we have
\begin{proposition}\label{prop:WFBasis}
For $z\in \calS_h$ and $j\in \{1,2,\ldots,n_z\}$,
let $\varphi_z^{(j)}$ be the piecewise constant function
\[
\varphi_z^{(j)}|_{K_v^{(i)}} = \delta_{v,z}\delta_{i,j}\qquad \forall v,z\in \calS_h,\ i=1,\ldots,n_v,\ j=1,\ldots,n_z.
\]
For an interior face split point $z$, define
\begin{align*}
\psi_z^{(3)} &= \varphi_z^{(3)}+\varphi_z^{(1)}+\varphi_z^{(2)},\\
\psi_z^{(4)} &= \varphi_z^{(4)}+\varphi_z^{(1)},\\
\psi_z^{(5)} &= \varphi_z^{(5)}+\varphi_z^{(2)},\\
\psi_z^{(6)} &= \varphi_z^{(6)}-\varphi_z^{(1)}-\varphi_z^{(2)}.
\end{align*}
For a boundary face split point $z$, define
\begin{align*}
\psi_z^{(3)} &= \varphi_z^{(3)}+\varphi_z^{(1)}+\varphi_z^{(2)}.
\end{align*}
Then $\{\psi_z^{(j)}\}$ is a basis of $\hY_h^{\rm WF}$.
\begin{proof}
The proof essentially follows from the same arguments as
Proposition \ref{prop:PSBasis}, noting that the number of given
$\psi_z^{(j)}$ is
\begin{align*}
4 |\calS_h^I|+|\calS_h^B| = \dim \hY_h^{\rm WF}
\end{align*}
by Proposition \ref{prop:WFDim}.
\end{proof}

\end{proposition}

As in the two-dimensional case,
Proposition \ref{prop:WFBasis} give
an algorithm
to construct the stiffness matrix for the Stokes problem
using the $\mbV_h^{\rm WF}\times \hY_h^{\rm WF}$
pair.  First, we construct the stiffness matrix
based on the $\mbV_h^{\rm WF}\times \pol_0(\WFT)$
pair:
\begin{align*}
\begin{pmatrix}
A & \tilde B\\
\tilde B^\intercal & 0
\end{pmatrix},
\end{align*}
and then perform elementary column
operations on the $\tilde B$.  

Let $\{\bphi^{(k)}\}_{k=1}^N$ denote a basis of $\mbV_h^{\rm WF}$
with $N = \dim \bV_h^{\rm PS}$
 Let $M = \dim \pol_0(\PST)$, the number of tetrahedra
in $\WFT$, and introduce the local-to-global label mapping
$\sigma:\calS_h\times \{1,\ldots,n_z\}\to \{1,2,\ldots,M\}$ such that
\[
\tilde B_{k,\sigma(z,j)} = -\int_\Omega (\Div \bphi^{(k)}) \varphi_z^{(j)}.
\]

\begin{mathframed}
{\bf Algorithm 1}\label{Alg.1}
\begin{enumerate}
\item Construct Worsey-Farin triangulation $\PST$

\item Construct $\tilde B\in \bbR^{N\times M}$ based on the $\mbV_h^{\rm WF}\times \pol_0(\WFT)$ pair.
\item Set $B = \tilde B$.

\item For each interior face split point $z\in \calS_h$ do the elementary column operations
\begin{align*}
B_{:,\sigma(z,3)} &=B_{:,\sigma(z,3)} +B_{:,\sigma(z,1)}+B_{:,\sigma(z,2)},\\
B_{:,\sigma(z,4)} &=B_{:,\sigma(z,4)} +B_{:,\sigma(z,1)},\\
B_{:,\sigma(z,5)} &=B_{:,\sigma(z,5)} +B_{:,\sigma(z,2)},\\
B_{:,\sigma(z,6)} &=B_{:,\sigma(z,6)} -B_{:,\sigma(z,1)}-B_{:,\sigma(z,2)}.
\end{align*}

\item For each boundary face split point $z\in \calS_h$ do the elementary column operation
\begin{align*}
B_{:,\sigma(z,3)} &=B_{:,\sigma(z,3)} +B_{:,\sigma(z,1)}+B_{:,\sigma(z,2)}.
\end{align*}

\item Delete columns $B_{:,\sigma(z,1)}$ and $B_{:,\sigma(z,2)}$ for each $z\in \calS_h$. \\

\end{enumerate}
\end{mathframed}

The stiffness matrix for the Stokes problem
based on the $\mbV_h^{\rm WF}\times \hY_h^{\rm WF}$ pair is then given by
\begin{align}
\begin{pmatrix}
A &  B\\
 B^\intercal & 0
\end{pmatrix},
\label{Eq:Stiffness_Matrix}
\end{align}


\section{Numerical Experiments}\label{numerical}

In this section, we perform some simple numerical experiments for the Stokes problem
on Powell--Sabin and Worsey--Farin splits.
We note standard theory shows that the velocity and pressure errors satisfy
\begin{align}
\label{ineqvub}
|\bu-\bu_h|_{H^1(\Omega)} &\le (1+\beta^{-1}) \inf_{\bv_h\in\bV_h} |\bv_h-\bu|_{H^1(\Omega)},\\
\label{ineqpub}
\|p-p_h\|_{L^2(\Omega)} &\le \inf_{q \in Y_h} \| p - q\|_{L^2(\Omega)} + \frac{\nu}{\beta}|\bu-\bu_h|_{H^1(\Omega)} ,
\end{align}
where either $\bV_h\times Y_h = \bV_h^{\rm PS}\times \mY_h^{\rm PS}$ or $\bV_h\times Y_h = \mbV_h^{\rm WF}\times \mY_h^{\rm WF}$, $\nu>0$ is the viscosity,
and $\beta$ is the inf-sup constant for the finite element pair $\bV_h\times Y_h$.

\subsection{The Stokes pair on Powell-Sabin Splits}
We consider the example such that the data
is taken to be 
$\Omega = (0,1)^2$, and the source function is chosen such that 
the exact velocity and pressure solutions for \eqref{eqn:3DProblem} are given respectively as
\begin{align}
\label{eqn:TestProblem}
\bu=\begin{pmatrix}
\pi\sin^2(\pi x_1)\sin(2\pi x_2) \\
-\pi\sin^2(\pi x_2)\sin(2\pi x_1)
\end{pmatrix},\quad
 p = \cos(\pi x_1)\cos(\pi x_2) .
\end{align}

Let $\mct$ be a Delaunay triangulation of $\Omega$ and $\calT_h^{PS}$ the corresponding Powell-Sabin global triangulation.


 The resulting errors, rates of convergence, and inf-sup constants are listed in 
Tables \ref{tableOne} and \ref{tableTwo} for viscosities $\nu = 1$ and $\nu = 10^{-2}$, respectively.
The results show that the $L^2$ pressure error and the $H^1$ velocity error converge  
with linear rate, the discrete velocity solution (and error) are independent of the viscosity $\nu$,
and the pressure error improves for small viscosity.
The experiments also show that the inf-sup constant does not deteriorate as the mesh
is refined with $\beta \approx 0.1$.
These results are in agreement with the theoretical estimates \eqref{ineqvub}--\eqref{ineqpub}




 In Tables \ref{tableThree} and \ref{tableFour}, we compute the right-hand side of \eqref{ineqpub} and 
\eqref{ineqvub}, respectively, and compare the data with the computed errors $\|p-p_h\|_{L^2(\Omega)}$
and $|\bu-\bu_h|_{H^1(\Omega)}$.  Again, the results are consistent with \eqref{ineqvub}--\eqref{ineqpub},
and they suggest that  the  term $|\bu-\bu_h|_{H^1(\Omega)}$ is
 the dominant term in the pressure error \eqref{ineqpub}.

%


\begin{table}
\caption{\label{tableOne}Errors and rates of convergence
for example \eqref{eqn:TestProblem} with $\nu=1$.}
 \begin{tabular}{||c c c c c c c||} 
 \hline
 $h$ & $\|\bu-\bu_h\|_{L^2(\Omega)}$ & rate & $\|p-p_h\|_{L^2(\Omega)}$ & rate & $\|\nab\cdot\bu_h\|_{L^2(\Omega)}$
 & $\beta$ \\ [0.5ex] 
 \hline
$2^{-2}$ &1.70E-01	&--		&5.26E 00	&--              &2.70E-14 &1.56E-01\\
$2^{-3}$ &5.66E-02	&1.587	&3.77E 00	&0.480       &6.65E-14 &1.38E-01\\
$2^{-4}$ &1.35E-02	&2.068	&1.68E 00	&1.166       &2.38E-13 &1.07E-01\\
$2^{-5}$ &3.35E-03	&2.011	&8.28E-01	&1.021       &8.38E-12 &1.06E-01\\
$2^{-6}$ &8.77E-04	&1.934	&4.25E-01	&0.962       &4.05E-10 	&9.34E-02\\
 \hline
\end{tabular}
\end{table}

\begin{table}
\caption{\label{tableTwo}Errors and rates of convergence
for example \eqref{eqn:TestProblem} with $\nu=10^{-2}$.}
 \begin{tabular}{||c c c c c c||} 
 \hline
 $h$ & $\|\bu-\bu_h\|_{L^2(\Omega)}$ & rate & $\|p-p_h\|_{L^2(\Omega)}$ & rate & $\|\nab\cdot\bu_h\|_{L^2(\Omega)}$\\ [0.5ex] 
 \hline
$2^{-2}$ &1.70E-01	&--		&1.02E-01	&--              &2.43E-14	\\
$2^{-3}$ &5.66E-02	&1.587	&5.79E-02	&0.816       &5.88E-14	\\
$2^{-4}$ &1.35E-02	&2.068	&2.76E-02	&1.069       &2.36E-13	\\
$2^{-5}$ &3.35E-03	&2.011	&1.37E-02	&1.010       &8.39E-12	\\
$2^{-6}$ &8.77E-04	&1.934	&6.96E-03	&0.977       &4.05E-10\\
 \hline
\end{tabular}
\end{table}

\begin{table}
\caption{\label{tableThree}Errors
 for example \eqref{eqn:TestProblem} with $\nu=10^{-2}$ and the RHS of \eqref{ineqpub} .}
 \begin{tabular}{||c c c c c c||} 
 \hline 
 $h$ & $\|p-p_h\|_{L^2(\Omega)}$ & $|\bu-\bu_h|_{H^1(\Omega)}$  & $\beta$&$\inf_{q \in \mY_h^{\rm PS}} \| p - q\|_{L^2(\Omega)}$ & RHS of \eqref{ineqpub}  \\ [0.5ex] 
 \hline
$2^{-2}$ &1.02E-01	&3.77E 00	&1.56E-01     &6.08E-02	&3.02E-01\\
$2^{-3}$ &5.79E-02	&2.17E 00	&1.38E-01     &2.77E-02	&1.84E-01\\
$2^{-4}$ &2.76E-02	&1.07E 00	&1.07E-01     &1.35E-02    &1.13E-01\\
$2^{-5}$ &1.37E-02	&5.32E-01	&1.06E-01     &6.61E-03	&5.67E-02\\
$2^{-6}$ &6.96E-03	&2.72E-01	&9.34E-02     &3.28E-03	&3.24E-02\\
 \hline
\end{tabular}
\end{table}
\begin{table}
\caption{\label{tableFour}Errors
 for example \eqref{eqn:TestProblem} with $\nu=10^{-2}$ and the RHS of \eqref{ineqvub} .}
 \begin{tabular}{||c c c c c||} 
 \hline 
 $h$ & $|\bu-\bu_h|_{H^1(\Omega)}$  & $\beta$&$\inf_{\bv_h\in\bV_h^{\rm PS}} |\bv_h-\bu|_{H^1(\Omega)}$ & RHS of \eqref{ineqvub}  \\ [0.5ex] 
 \hline
$2^{-2}$ &3.77E 00	&1.56E-01     &3.08E 00	&2.28E+01\\
$2^{-3}$ &2.17E 00	&1.38E-01     &1.63E 00	&1.34E+01\\
$2^{-4}$ &1.07E 00	&1.07E-01     &8.04E-01    &8.31E 00\\
$2^{-5}$ &5.32E-01	&1.06E-01     &4.06E-01	&4.23E 00\\
$2^{-6}$ &2.72E-01	&9.34E-02     &2.05E-01	&2.39E 00\\
 \hline
\end{tabular}
\end{table}

\subsection{The Stokes pair on Worsey-Farin Splits}
We consider the example such that the data
is taken to be 
$\Omega = (0,1)^3$, and the source function is chosen such that 
the exact velocity and pressure solutions for \eqref{eqn:3DProblem} are given respectively as
\begin{align}
\label{eqn:TestProblem1}
\bu=\begin{pmatrix}
\pi\sin^2(\pi x_1)\sin(2\pi x_2) \\
-\pi\sin^2(\pi x_2)\sin(2\pi x_1)\\
0
\end{pmatrix},\quad
 p = \cos(\pi x_1)\cos(\pi x_2)\cos(\pi x_3) .
\end{align}

 Let $\mct$ be a Delaunay triangulation of $\Omega$ and $\calT_h^{WF}$ be the corresponding Worsey-Farin global triangulation.

 The resulting rates of convergence of the numerical experiments for viscosities 
 $\nu = 1$ and $\nu = 10^{-3}$ are listed in
Tables \ref{tableFive} and \ref{tableSix}, respectively.  We also 
state the computed inf-sup constant on these meshes, and the results
show that that it stays uniformly bounded from below with $\beta \approx 0.13$ on all meshes.
The stated errors, especially those in Table \ref{tableFive}, indicate that the rates of convergence
are still in the preasymptotic regime.  On the other hand, 
for small viscosity value $\nu=10^{-3}$, Table \ref{tableSix} shows
that the pressure error converges with linear rate.  This behavior suggests
that the velocity error is the dominating term in \eqref{ineqpub}.

To verify this claim, we explicitly compute
the right-hand side of \eqref{ineqpub} and \eqref{ineqvub}
and report the results in Tables \ref{tableSeven}
and \ref{tableEight}, respectively.
The results show that indeed $|\bu-\bu_h|_{H^1(\Omega)}$
is the dominating term in the pressure error \eqref{ineqpub}.

\begin{table}
\caption{\label{tableFive}Errors and rates of convergence
for example \eqref{eqn:TestProblem1} with $\nu=1$.}
 \begin{tabular}{||c c c c c c c||} 
 \hline
 $h$ & $\|\bu-\bu_h\|_{L^2(\Omega)}$ & rate & $\|p-p_h\|_{L^2(\Omega)}$ & rate & $\|\nab\cdot\bu_h\|_{L^2(\Omega)}$ & $\beta$\\ [0.5ex] 
 \hline
$1/2      $ &1.29E 00	&--		&9.81E 00	&--              &5.07E-14 	&1.31E-01\\
$1/4      $ &8.58E-01	&0.588	&19.4E 00	&-0.98        &5.20E-13 	&1.31E-01\\
$1/8      $ &3.93E-01	&1.286	&16.6E 00	&0.414       &2.68E-12            &1.32E-01\\
$1/16$     &1.32E-01    &1.573  & 10.5E 00  &0.667            &4.10E-12             &1.32E-01\\
$1/32$     &3.69E-02    &1.839  & 5.75E 00  &0.872            &4.32E-12             &1.32E-01\\
$1/48$     &1.68E-02    &1.941  & 3.93E 00  &0.936            &6.07E-12             &1.32E-01\\
 \hline
\end{tabular}
\end{table}

\begin{table}
\caption{\label{tableSix}Errors and rates of convergence
for example \eqref{eqn:TestProblem1} with $\nu=10^{-3}$.}
 \begin{tabular}{||c c c c c c||} 
 \hline
 $h$ & $\|\bu-\bu_h\|_{L^2(\Omega)}$ & rate & $\|p-p_h\|_{L^2(\Omega)}$ & rate & $\|\nab\cdot\bu_h\|_{L^2(\Omega)}$ \\ [0.5ex] 
 \hline
$1/2      $ &1.29E 00	&--		&1.33E-01	&--          &1.28E-15\\
$1/4      $ &8.58E-01	&0.588	&6.97E-02	&0.932       &3.43E-14\\
$1/8      $ &3.93E-01	&1.286	&3.70E-02	&0.911       &3.22E-13            	\\
$1/16    $ &1.32E-01	&1.574	&1.91E-02 	&0.953       &6.40E-13                   \\
$1/32    $ &3.69E-02	&1.838	&9.68E-03 	&0.980       &9.52E-13                   \\
$1/48    $ &9.63E-03	&1.940	&4.89E-03 	&0.983       &1.03E-12                   \\
 \hline
\end{tabular}
\end{table}
\begin{table}
\caption{\label{tableSeven}Errors
 for example \eqref{eqn:TestProblem1} with $\nu=10^{-3}$ and the RHS of \eqref{ineqpub} .}
 \begin{tabular}{||c c c c c c||} 
 \hline 
 $h$ & $\|p-p_h\|_{L^2(\Omega)}$ & $|\bu-\bu_h|_{H^1(\Omega)}$  & $\beta$&$\inf_{q \in \mY_h^{\rm WF}} \| p - q\|_{L^2(\Omega)}$ & RHS of \eqref{ineqpub}  \\ [0.5ex] 
 \hline
$1/2      $ &1.33E-01	&1.07E+01	&1.31E-01     &5.00E-01	&5.81E-01\\
$1/4      $ &6.97E-02	&8.50E 00	&1.31E-01     &6.70E-02	&1.31E-01\\
$1/6      $ &4.81E-02	&6.82E 00	&1.32E-01     &4.43E-02    &9.59E-02\\
$1/8      $ &3.70E-02	&5.60E 00	&1.32E-01     &3.30E-02	&7.54E-02\\
$1/10    $ &3.02E-02	&4.71E 00	&1.32E-01    &2.63E-02	&6.19E-02\\
$1/12    $ &2.55E-02	&4.06E 00	&1.32E-01    &2.19E-02	&5.26E-02\\
 \hline
\end{tabular}
\end{table}
\begin{table}
\caption{\label{tableEight}Errors
 for example \eqref{eqn:TestProblem1} with $\nu=10^{-3}$ and the RHS of \eqref{ineqvub} .}
 \begin{tabular}{||c c c c c||} 
 \hline 
 $h$ & $|\bu-\bu_h|_{H^1(\Omega)}$  & $\beta$&$\inf_{\bv_h\in\mbV_h^{\rm WF}} |\bv_h-\bu|_{H^1(\Omega)}$ & RHS of \eqref{ineqvub}  \\ [0.5ex] 
 \hline
$1/2      $ &1.07E+01	&1.31E-01     &9.99E 00	&8.62E+01\\
$1/4      $ &8.50E 00	&1.31E-01     &5.85E 00	&5.05E+01\\
$1/6      $ &6.82E 00	&1.32E-01     &4.08E 00    &3.49E+01\\
$1/8      $ &5.60E 00	&1.32E-01     &3.10E 00	&2.65E+01\\
$1/10    $ &4.71E 00	&1.32E-01    &2.49E 00	&2.13E+01\\
$1/12    $ &4.06E 00	&1.32E-01    &2.08E 00	&1.78E+01\\
 \hline
\end{tabular}
\end{table}
\subsection{Iterated Penalty Method for $\bf{(P_1,P_0)}$ pair on Worsey-Farin Splits}
We consider the example such that the data
is taken to be 
$\Omega = (0,1)^3$, and the source function is chosen such that 
the exact velocity and pressure solutions for \eqref{eqn:3DProblem} are given respectively as
\begin{align}
\label{eqn:TestProblem2}
\bu(x,y,z)=\nab\times\begin{pmatrix}
0 \\
g\\
g
\end{pmatrix},\quad
 p = \frac{1}{9}\frac{\partial^2 g}{\partial x \partial y},
\end{align}
where 
\begin{equation*}
g = g(x,y,z) = 2^{12} (x-x^2)^2  (y-y^2)^2  (z-z^2)^2.
\end{equation*}
 Similar to the previous section, we let $\mct$ be a Delaunay triangulation of $\Omega$ and $\calT_h^{\rm WF}$
 be the corresponding Worsey-Farin global triangulation.

 The iterated penalty method \cite{BSBook} applied to the Stokes equations with 
$\mathring{\bV}_h = \mathring{\bV}_h^{\rm WF}$ reads : Let
 $\bu_h^0 = \bm 0$ and $\rho,\gamma>0$ be parameters. For $n\geq 1$, $\bu_h^n$ is recursively defined to be the solution to the variational formulation

\begin{equation}\label{Eq:IPM}
\nu(\nabla\bu_h^n,\nabla{\bv}) + \gamma(\nabla\cdot{\bv},\nabla\cdot{\bu_h^n})=({{\bm f},{\bv}}) - (\sum_{i = 0}^{n-1}\rho\nabla\cdot{\bu_h^i},\nabla\cdot{\bv}), \quad \forall \bv\in \mathring{\bV}_h^{\rm WF}.
\end{equation}

 It is shown in \cite{BSBook} that $\lim_{n\to\infty} \bu_h^n = \bu_h$ and $\lim_{n\to\infty}\sum_{i = 0}^{n} \rho\nabla\cdot\bu_h^i = p_h$. Also, 
\cite{BSBook} suggests to use $\|\nab\cdot\bu_h^n\|_{L^2(\Omega)}$ as a stopping criterion since the difference error between $\bu_h^n, \bu_h$
 is given by
 \begin{equation*}
\|\bu_h^n - \bu_h\|_{L^2(\Omega)} \leq C\|\nab\cdot\bu_h^n\|_{L^2(\Omega)}.
\end{equation*}

 The resulting rates of convergence of the numerical experiment are listed in 
Tables \ref{tableNine}. The errors $\|\bu-\bu_h^n\|_{L^2(\Omega)}$ and 
$\|p-p_h^n \|_{L^2(\Omega)}$ are computed 
with $\|\nab\cdot\bu_h^n\|_{L^2(\Omega)} \leq 10^{-7}$ and $\gamma = \rho = 100$.

\begin{table}
\caption{\label{tableNine}Errors and rates of convergence
for example \eqref{eqn:TestProblem2} with $\nu=1$.}
 \begin{tabular}{||c c c c c c c||} 
 \hline
 $h$ & $\|\bu-\bu_h^n\|_{L^2(\Omega)}$ & rate &$|\bu-\bu_h^n|_{H^1(\Omega)}$ & rate & $\|p-p_h^n\|_{L^2(\Omega)}$ & rate  \\ [0.5ex] 
 \hline
    $2^{-2}$ &  1.11768  &  -        &  11.55063 &  -        & 25.32256  &  -          \\
    $2^{-3}$ &  0.48896  &  1.19273  &  7.53829  &  0.61566  & 22.35349  &  0.17992\\
    $2^{-4}$ &  0.15482  &  1.65908  &  4.15598  &  0.85905  & 13.67635  &  0.70882\\
    $2^{-5}$ &  0.04176  &  1.89040  &  2.13224  &  0.96282  &  7.24129  &  0.91736\\
    $1/48$   &  0.01881  &  1.96680  &  1.42643  &  0.99145  &  4.88909  &  0.96875\\
 \hline
\end{tabular}
\end{table}

 In Table \ref{tableTen} we provide a time comparison between the (IPM) and the method described in Algorithm 1.  All of the timings were done on a machine with a single 3.60 GHz Intel Core i9-9900K processor with 128 GB of 2400 MHz DDR4 memory. 
 
 For the iterated penalty method, we select $\gamma=\rho=100$ and terminate iterations once $\|\nab\cdot\bu_h^n\|_{L^2(\Omega)} \leq 10^{-7}$.  At each iteration of IPM, equation~\eqref{Eq:IPM} must be resolved.  In our work, this vector Poisson problem type problem is solved via a conjugate gradient method with an algebraic multigrid (AMG) preconditioner.  It is well known that optimal multigrid methods can be designed for this class of problems~\cite{trottenberg2001multigrid}. 
 
 For Algorithm 1, we instead work with the full discretization matrix~\eqref{Eq:Stiffness_Matrix}, which results in a symmetric indefinite linear system.  To efficiently solve this saddle point problem, a block preconditioned Krylov subspace method is used~\cite{StokesSolver}.  In particular, we precondition the stiffness matrix~\eqref{Eq:Stiffness_Matrix} by the block diagonal matrix
 \[
\begin{pmatrix}
 A & 0
 \\
 0 & S
\end{pmatrix},
 \]
 where $S= B^\intercal A^{-1}B$ is the Schur complement, and the flexible GMRES method as an outer iteration.  The solver is terminated once the Euclidean norm of the residual is less than or equal to $10^{-8}.$  For simplicity, the inner preconditioners are as follows: we use a preconditioner to $A$ for approximating $A^{-1}$, and the Schur complement is approximated as $S \approx B^\intercal \mathrm{diag}(A^{-1})B$.  It should be noted that this is not the only choice for block-type preconditioning of the Stokes problem (e.g., see~\cite{StokesSolver2}); however, our numerical experiments indicate that we can repurpose existing preconditioners to efficiently solve linear systems that arise from the proposed finite element discretization.
 
 From Table \ref{tableTen}, it is evident that for a fixed mesh spacing, the total run times for both methods are competitive, with the IPM technique being slightly slower for larger meshes.   
 We find that in the case of the linear Stokes problem, by utilizing preexisting preconditioners, both approaches provide similar accuracy and time to solution metrics.  
 
 To obtain a Reynolds robust, optimally scaling preconditioner, special multigrid techniques may have to be investigated~\cite{SMAI-JCM_2021__7__75_0,PCPATCH}.

\begin{table}
\caption{\label{tableTen}Time comparison between (IPM) and Algorithm 1 
for example \eqref{eqn:TestProblem2} with $\nu=1$.}
 \begin{tabular}{||c c c||} 
 \hline
 $h$ & (IPM) & Algorithm 1  \\ [0.5ex] 
 \hline
    $2^{-2}$ &  5.08E+00 &  3.35E+00   \\
    $2^{-3}$ &  1.68E+01 &  2.93E+01   \\
    $2^{-4}$ &  4.80E+02 &  2.60E+02   \\
    $2^{-5}$ &  2.39E+03 &  9.37E+02   \\
    $1/48$   &  7.31E+03 &  6.45E+03   \\
 \hline
\end{tabular}
\end{table}

\bibliographystyle{abbrv}
\bibliography{refs}










\end{document}